\documentclass[11pt]{amsart}
\usepackage[margin=1.15in]{geometry}
\usepackage{amscd, amssymb, amsmath, wasysym}
\usepackage{graphicx}
\usepackage{amsfonts}
\usepackage{mathrsfs}    
\usepackage{eucal}     
\usepackage{latexsym}   
\usepackage{verbatim}   
\usepackage[all]{xy}   
\usepackage[dvipsnames]{xcolor}
\usepackage{bookmark}
\usepackage{subcaption}

 \usepackage{hyperref}
 \hypersetup{
     colorlinks=true,
     linkcolor=NavyBlue,
     filecolor=NavyBlue,
     citecolor = TealBlue,
     urlcolor=magenta,
  }

\newcounter{thmcounter}

\numberwithin{equation}{section}
\numberwithin{thmcounter}{section}

\newtheorem{theorem}[thmcounter]{Theorem}
\newtheorem{proposition}[thmcounter]{Proposition}
\newtheorem{lemma}[thmcounter]{Lemma}
\newtheorem{corollary}[thmcounter]{Corollary}

\theoremstyle{definition}

\newtheorem{example}[thmcounter]{Example}

\newtheorem{remark}[thmcounter]{Remark}

\newtheoremstyle{claim}{9pt}{3pt}{}{\parindent}{\bf}{.}{1em}{}

\theoremstyle{claim}

\newenvironment{namelist}[1]{%
\begin{list}{}
{
\settowidth{\labelwidth}{#1}%
\setlength{\labelsep}{0.3em}%
\setlength{\leftmargin}{\labelwidth}%
\addtolength{\leftmargin}{\labelsep}}}{%
\end{list}}

\newcommand{\nZ}{\mathbf{Z}}                   
               
\newcommand{\nC}{\mathbf{C}}                     
           
\newcommand{\kk}{\mathbf{k}}         
     
\newcommand{\nP}{\mathbf{P}}

\newcommand{\sF}{\mathscr{F}}
\newcommand{\sG}{\mathscr{G}}
\newcommand{\sO}{\mathscr{O}}                  

\newcommand{\sI}{\mathscr{I}}

\newcommand{\mf}[1]{\mathfrak{#1}}

\DeclareMathOperator{\Cliff}{Cliff}

\DeclareMathOperator{\coker}{coker}

\DeclareMathOperator{\gon}{gon}

\DeclareMathOperator{\im}{im}

\DeclareMathOperator{\id}{id}

\DeclareMathOperator{\Sing}{Sing}

\DeclareMathOperator{\sgn}{sgn}           

\DeclareMathOperator{\reg}{reg}

\DeclareMathOperator{\rank}{rank}             

\newcounter{rkcounter}             % set remark counter
\begin{document}

\title[Syzygies of tangent surfaces and K3 carpets]{Syzygies of tangent developable surfaces\\ and K3 carpets via secant varieties}

\author{Jinhyung Park}
\address{Department of Mathematical Sciences, KAIST, 291 Daehak-ro, Yuseong-gu, Daejeon 34141, Republic of Korea}
\email{parkjh13@kaist.ac.kr}

\date{\today}
%\dedicatory{Dedicated to}
%\keywords{tangent developable surfaces, K3 carpets, secant varieties, Green's conjecture, syzygies}
\thanks{J. Park was partially supported by the National Research Foundation (NRF) funded by the Korea government (MSIT) (NRF-2021R1C1C1005479).}

\begin{abstract} 
We give simple geometric proofs of Aprodu--Farkas--Papadima--Raicu--Weyman's theorem on syzygies of tangent developable surfaces of rational normal curves and Raicu--Sam's result on syzygies of K3 carpets.  As a consequence, we obtain a quick proof of Green's conjecture for general curves of genus $g$ over an algebraically closed field $\mathbf{k}$ with $\operatorname{char}(\mathbf{k}) = 0$ or $\operatorname{char}(\mathbf{k}) \geq \lfloor (g-1)/2 \rfloor$. We also show the arithmetic normality of tangent developable surfaces of arbitrary smooth projective curves of large degree.
\end{abstract}

\maketitle

%\tableofcontents

%%%%%%%%%%%%%%%%%%%%%%%%%%%%%%%%%%%%%%%%%%%%%%%%%%
\section{Introduction}
%%%%%%%%%%%%%%%%%%%%%%%%%%%%%%%%%%%%%%%%%%%%%%%%%%

Let $C $ be a smooth projective curve of genus $g \geq 3$ over the field $\nC$ of complex numbers. If $\Cliff(C) \geq 1$, i.e., $C$ is nonhyperelliptic, then $K_C$ is very ample. In this case, Noether's theorem says that the canonical curve $C \subseteq \nP^{g-1}$ is projectively normal. Petri's theorem states that if $\Cliff(C) \geq 2$, then the defining ideal $I_{C|\nP^{g-1}}$ of $C$ in $\nP^{g-1}$ is generated by quadrics. To generalize classical theorems of Noether and Petri, in the early 1980's, Green \cite[Conjecture 5.1]{Green} formulated a very famous conjecture that predicts 
$$
K_{p,2}(C, K_C)=0 ~~\text{ for $0 \leq p \leq \Cliff(C)-1$}.
$$
By Green--Lazarsfeld's nonvanishing theorem \cite[Appendix]{Green}, we have $K_{p,2}(C, K_C) \neq 0$ for $\Cliff(C) \leq p \leq g-3$. Then Green's conjecture determines the shape of the minimal free resolution of the canonical ring $R(C, K_C)=\bigoplus_{m \in \nZ} H^0(C, mK_C)$ (see \cite[Remark 4.19]{AN}). Although the conjecture is still open, Voisin \cite{Voisin1,  Voisin2} resolved the general curve case in the early 2000's. To prove generic Green's conjecture, it suffices to exhibit one curve of each genus for which the assertion holds. If $S \subseteq \nP^g$ is a K3 surface of degree $2g-2$, then its general hyperplane section is a canonical curve $C \subseteq \nP^{g-1}$ and $K_{p,2}(S, \sO_S(1)) = K_{p,2}(C, K_C)$. Based on the Hilbert scheme interpretation of Koszul cohomology, Voisin accomplished sophisticated cohomology computations on Hilbert schemes of K3 surfaces to show $K_{p,2}(S, \sO_S(1)) = 0$. For an introduction to Voisin's work, see \cite{AN}. Recently, Kemeny \cite{Kemeny} gave a simpler proof of Voisin's theorem for even genus case and a steamlined version of her arguments for odd genus case.  

\medskip

Prior to Voisin's work, O'Grady and Buchweitz--Schreyer independently observed that one could use the \emph{tangent developable surface} $T \subseteq \nP^g$ of a rational normal curve of degree $g$ to solve Green's conjecture for general curves of genus $g$ (see \cite{Eisenbud}). Note that $T \subseteq \nP^g$ is arithmetically Cohen--Macaulay as for a K3 surface. One can actually view $T$ as a degenerate K3 surface  (see \cite[Remark 6.6]{AFPRW1}). A general hyperplane section of $T$ is a canonically embedded $g$-cuspidal rational curve $\overline{C} \subseteq \nP^{g-1}$ of degree $2g-2$, which is degenerated to a general canonical curve $C \subseteq \nP^{g-1}$ with $\Cliff(C) = \lfloor (g-1)/2 \rfloor$. By the upper semicontinuity of graded Betti numbers, $K_{p,2}(\overline{C}, \sO_{\overline{C}}(1)) = K_{p, 2}(T, \sO_T(1)) = 0$ implies $K_{p,2}(C, K_C)=0$  (see \cite[Section 6]{AFPRW1}). The required vanishing of $K_{p,2}(T, \sO_T(1))$ for confirming generic Green's conjecture was finally established by Aprodu--Farkas--Papadima--Raicu--Weyman \cite{AFPRW1} just a few years ago. Their important result gives not only an alternative proof of generic Green's conjecture but also an extension to positive characteristic. This circle of ideas is surveyed in  \cite{EL}. In this paper, we give a simple geometric proof of the main result of \cite{AFPRW1}. 

\begin{theorem}\label{thm:AFPRW}
Let $T \subseteq \nP^g$ be the tangent developable surface of a rational normal curve of degree $g \geq 3$ over an algebraically closed field $\kk$ with $\operatorname{char}(\kk) = 0$ or $\operatorname{char}(\kk) \geq (g+2)/2$. Then
$$
K_{p, 2}(T, \sO_T(1)) = 0~~\text{ for $0 \leq p \leq \lfloor (g-3)/2 \rfloor$}.
$$
\end{theorem}

The proof of Theorem \ref{thm:AFPRW} in \cite{AFPRW1} goes as follows. We have a short exact sequence
$$
\xymatrix{
0 \ar[r] & \sO_T \ar[r] & \nu_* \sO_{\widetilde{T}} \ar[r] & \omega_{\nP^1} \ar[r] & 0,
}
$$
where $\nu \colon \widetilde{T} = \nP^1 \times \nP^1 \to T$ is a resolution of singularities. Then it is elementary to see that 
$$
K_{p,2}(T, \sO_T(1)) = \coker\big(K_{p,1}(T, \nu_* \sO_{\widetilde{T}}; \sO_T(1)) \xrightarrow{~\gamma~} K_{p,1}(\nP^1, \omega_{\nP^1}; \sO_{\nP^1}(g)) \big).
$$
Let $U:=H^0(\nP^1, \sO_{\nP^1}(1))$, $V:=D^{p+2} U$, $W:=D^{2p+2} U$, and $q:=g-p-3$. The authors of \cite{AFPRW1} devoted considerable effort to show that $\gamma$ arises as the composition
\begin{equation}\label{eq:gammaintro}
S^{q} V \otimes W \xrightarrow{\id_{S^{q}V} \otimes \Delta}  S^{q} V \otimes \wedge^2 V \xrightarrow{~\delta~}  \ker(S^{q+1}V \otimes V \xrightarrow{~\delta~} S^{q+2}V),
\end{equation}
where $\Delta$ is the \emph{co-Wahl map} and $\delta$ is the \emph{Koszul differential}. To achieve this,  they established an explicit characteristic-free Hermite reciprocity for $\mf{sl}_2$-representations, and they carried out complicated algebraic computations. Now, $K_{p,2}(T, \sO_T(1))$ is the homology of a complex
$$
S^q V \otimes W \xrightarrow{~\gamma~} S^{q+1}V \otimes V \xrightarrow{~\delta~} S^{q+2} V.
$$
This homology, denoted by $W_q(V, W)$, is the degree $q$ piece of the \emph{Koszul module} (or \emph{Weyman module}) associated to $(V, W)$. It is enough to prove that 
\begin{equation}\label{eq:vanKosmod}
W_q(V, W) = 0~~\text{ for $q \geq p$}.
\end{equation}
The vanishing result (\ref{eq:vanKosmod}) was first proved in characteristic zero in  \cite{AFPRW2}  by an application of Bott vanishing, and the argument is extended in \cite{AFPRW1} to positive characteristics. 

\medskip

Our strategy to prove Theorem \ref{thm:AFPRW} is essentially the same as of \cite{AFPRW1}, but our geometric approach utilizing the \emph{secant variety} $\Sigma \subseteq \nP^g$ of a rational normal curve $C$ of degree $g$ provides a substantial simplification of the proof. The tangent surface $T$ is a Weil divisor on $\Sigma$, and $\widetilde{T}$ is a Cartier divisor on $B$, where $\beta \colon B \to \Sigma$ is the blow-up of $\Sigma$ along $C$ with the exceptional divisor $Z=\nP^1 \times \nP^1$. Letting $M_H := \beta^* M_{\sO_{\Sigma}(1)}$, we realize (\ref{eq:gammaintro}) as maps induced in cohomology of vector bundles on $B$ (see (\ref{eq:coWahl+Koszul})):
$$
H^1(\widetilde{T}, \wedge^{p+2} M_H|_{\widetilde{T}}) \xrightarrow{~\alpha~} H^2(B, \wedge^{p+2} M_H \otimes \omega_B(Z)) \xrightarrow{~\delta~} H^2(Z, \omega_{\nP^1} \boxtimes \wedge^{p+2} M_L \otimes \omega_{\nP^1}).
$$
This was previously asked in \cite[last paragraph in p.666]{AFPRW1}. There is a rank two vector bundle $E$ on $\nP^2$ such that $B=\nP(E)$. If $\pi \colon B \to \nP^2$ is the canonical fibration, then $Q:=\pi(\widetilde{T})$ is a smooth conic and $\widetilde{T}=\nP(E|_Q)$. It is easy to check that $\alpha = \id_{S^qV} \otimes \Delta$, where $\Delta$ is the dual of the restriction map $H^0(\nP^2, \sO_{\nP^2}(p+1)) \longrightarrow H^0(Q, \sO_Q(p+1))$. 
Put $M_E:=\pi_* M_H$ and $\sigma:=\pi|_Z$. The map $\delta$ is naturally factored as
\begin{footnotesize}
$$
\xymatrixrowsep{0.15in}
\xymatrixcolsep{0.5in}
\xymatrix{
H^2(B, \wedge^{p+2} M_H \otimes \omega_B(Z)) \ar@{=}[d] \ar[r]^-{\id_{S^q V} \otimes \iota} & H^2(Z, \sigma^* \wedge^{p+2} M_E \otimes (\omega_{\nP^1} \boxtimes \omega_{\nP^1})) \ar@{=}[d] \ar[r]^-{m \otimes \id_V} & H^2(Z, \omega_{\nP^1} \boxtimes \wedge^{p+2} M_L \otimes \omega_{\nP^1})\ar@{=}[d] \\
S^q V \otimes \wedge^2 V & S^q V \otimes V \otimes V & S^{q+1} V \otimes V,
}
$$
\end{footnotesize}\\[-5pt]
where $\iota$ is the canonical injection and $m$ is the multiplication map identified with
$$
H^1(\nP^{p+2} \times \nP^1, \sO_{\nP^{p+2}}(q) \boxtimes \omega_{\nP^1}(-p-2)) \longrightarrow H^1(\nP^{p+2} \times \nP^1, \sO_{\nP^{p+2}}(q+1) \boxtimes \omega_{\nP^1}).
$$
Without any lengthy computation, we quickly obtain the key part of \cite{AFPRW1} -- the descriptions of the maps in (\ref{eq:gammaintro}) (see Lemma \ref{lem:gamma_{p_2}}). This provides a conceptual explanation of the difficult computation in \cite{AFPRW1}  and a geometric understanding of syzygies of $T$ in $\nP^g$ as expected in \cite[last paragraph in p.666]{AFPRW1}.
Next, regarding $V=H^0(\nP^{p+2}, \sO_{\nP^{p+2}}(1))$, we give a direct proof of (\ref{eq:vanKosmod}) using vector bundles on $\nP^{p+2}$. This part of the proof is largely equivalent to the original proof in \cite{AFPRW1, AFPRW2}.\footnote{After writing the paper, the author learned from Claudiu Raicu that a similar argument proving (\ref{eq:vanKosmod}) directly on projective spaces is given in his lecture notes \cite{Raicu} based on Robert Lazarsfeld's suggestion.} 
Put $M_V:=M_{\sO_{\nP^{p+2}}(1)}$. Then (\ref{eq:gammaintro}) can be identified with
$$
W \otimes H^0(\nP^{p+2}, \sO_{\nP^{p+2}}(q)) \longrightarrow \wedge^2 V  \otimes H^0(\nP^{p+2}, \sO_{\nP^{p+2}}(q)) \longrightarrow H^0(\nP^{p+2}, M_V(q+1)).
$$
Under the assumption $\operatorname{char}(\kk) = 0$ or $\operatorname{char}(\kk) \geq (g+2)/2$, we show that $W \otimes \sO_{\nP^{p+2}} \to M_V(1)$ is surjective and its kernel is $(p+1)$-regular in the sense of Castelnuovo--Mumford. This implies (\ref{eq:vanKosmod}). Here the characteristic assumption  plays a crucial role (see Remark \ref{rem:charassump}).

\medskip

It is worth noting that the characteristic assumption on the field $\kk$ in Theorem \ref{thm:AFPRW} cannot be improved. If $2 \leq \operatorname{char}(\kk) \leq (g+1)/2$, then $K_{\lfloor (g-3)/2 \rfloor, 2}(T, \sO_T(1)) \neq 0$ (see \cite[Remark 5.17]{AFPRW1}). Theorem \ref{thm:AFPRW} implies generic Green's conjecture for general curve of genus $g$ when $\operatorname{char}(\kk) = 0$ or $\operatorname{char}(\kk) \geq (g+2)/2$, but this is not optimal. Raicu--Sam \cite[Theorem 1.5]{RS2} recently obtained a sharp result that Green's conjecture holds for general curves of genus $g$ when $\operatorname{char}(\kk) = 0$ or $\operatorname{char}(\kk) \geq \lfloor(g-1)/2 \rfloor$. This confirms a conjecture of Eisenbud--Schreyer \cite[Conjecture 1.1]{ES}. For the failure of Green's conjecture in small characteristic, see \cite{ES, Schreyer2}. It has long been known that generic Green's conjecture would follow from the canonical ribbon conjecture \cite{BE}. A \emph{canonical ribbon} is a hyperplane section of a \emph{K3 carpet} $X=X(a,b) \subseteq \nP^{a+b+1}$ for integers $b \geq a \geq 1$, which is a unique double structure on a rational normal surface scroll $S(a,b) \subseteq \nP^{a+b+1}$ of type $(a,b)$ such that $\omega_X = \sO_X$ and $h^1(X, \sO_X)=0$ (see \cite[Theorem 1.3]{GP}). The K3 carpet $X$ is degenerated to a K3 surface of degree $2(a+b)$, and a canonical ribbon is degenerated to a general canonical curve of genus $a+b+1$ with $\Cliff(C) = a$. By extending algebraic arguments of \cite{AFPRW1}, Raicu--Sam  \cite[Theorem 1.1]{RS2} proved $K_{p,2}(X, \sO_{X}(1))=0$ for $0 \leq p \leq a$ when  $\operatorname{char}(\kk) = 0$ or $\operatorname{char}(\kk) \geq a$. This implies the canonical ribbon conjecture and hence generic Green's conjecture (see \cite[Section 6]{RS2}). However, for settling Eisenbud--Schreyer's conjecture, we only need to consider the case $a=\lfloor (g-1)/2 \rfloor$ and $b=\lfloor g/2 \rfloor$, and we recover \cite[Theorem 1.1]{RS2} for this case here.

\begin{theorem}\label{thm:K3carpet}
Let $X=X(\lfloor (g-1)/2 \rfloor, \lfloor g/2 \rfloor) \subseteq \nP^g$ be a K3 carpet with $g \geq 3$ over an algebraically closed field $\kk$ with $\operatorname{char}(\kk) = 0$ or $\operatorname{char}(\kk) \geq \lfloor (g-1)/2 \rfloor$ and $\operatorname{char}(\kk) \neq 2$. Then
$$
K_{p, 2}(X, \sO_X(1)) = 0~~\text{ for $0 \leq p \leq \lfloor (g-3)/2 \rfloor$}.
$$
In particular, Green's conjecture holds for general curves of genus $g$ over $\kk$.
\end{theorem}

Recall that Schreyer \cite{Schreyer2} verified Green's conjecture for every curve of genus $g \leq 6$ over an algebraically closed field of arbitrary characteristic. In the theorem, if $g \geq 7$, then the condition $\operatorname{char}(\kk) \neq 2$ is redundant. Our proof of Theorem \ref{thm:K3carpet} is essentially different from that of \cite{RS2} but surprisingly the same as that of Theorem \ref{thm:AFPRW}. The key point is that the K3 carpet $X$ in the theorem is a Weil divisor linearly equivalent to the tangent developable surface $T$ on the secant variety $\Sigma$. Thus $X$ is degenerated to the tangent surface $T$. The proof of Theorem \ref{thm:AFPRW} works for $X$, and the characteristic assumption on $\kk$ for (\ref{eq:vanKosmod}) with another $W$ can be improved. Consequently, we obtain a quick proof of generic Green's conjecture. 

\medskip

In view of Theorem \ref{thm:AFPRW}, it is quite natural to study syzygies of tangent developable surfaces of smooth projective curves of genus $g \geq 1$. As a first step, we show the arithmetic normality, and compute the Castelnuovo--Mumford regularity.

\begin{theorem}\label{thm:arithnorm}
Let $C$ be a smooth projective curve of genus $g \geq 1$ over an algebraically closed field of characteristic zero, $L$ be a line bundle on $C$ with $\deg L \geq 4g+3$, and $T$ be the tangent developable surface of $C$ embedded in $\nP^r$ by  $|L|$. Then $T \subseteq \nP^r$ is arithmetically normal but not arithmetically Cohen--Macaulay, and $H^i(T, \sO_T(m))=0$ for $i>0, m>0$ but $H^1(T, \sO_T) \neq 0, H^2(T, \sO_T) \neq 0$. In particular, $\reg \sO_T = 3$.
\end{theorem}

To prove the theorem, we use methods for secant varieties developed in \cite{ENP1}, and we show that the dualizing sheaf $\omega_T$ is trivial (Proposition \ref{prop:omegaS}). The hard part is to check the $2$-normality of $T \subseteq \nP^r$, which is turned out to be equivalent to $H^1(C \times C, (L \boxtimes L)(-3D)) = 0$, where $D$ is the diagonal of $C \times C$. This cohomology vanishing was established in \cite{BEL} when $\deg L \geq 4g+3$. This degree condition is optimal. In fact, if $\deg L = 4g+2$, then $T \subseteq \nP^r$ is arithmetically normal if and only if $C$ is neither elliptic nor hyperelliptic (see Remark \ref{rem:degL=4g+2}). We will further discuss syzygies of tangent developable surfaces in Remark \ref{rem:syztansurf}.

\medskip

After setting notations and presenting basic facts in Section \ref{sec:prelim}, we prove Theorems \ref{thm:AFPRW} and \ref{thm:K3carpet} in Section \ref{sec:proof1} and Theorem \ref{thm:arithnorm} in Section \ref{sec:proof2}. We work over an algebraically closed field $\kk$.

\subsection*{Acknowledgements}
This paper grew out from attempts to understand the algebraic arguments of \cite{AFPRW1} and \cite{RS2} in a geometric way, so the present author wishes to express his deep gratitude to the authors of those papers, particularly to Claudiu Raicu for valuable comments. The author is grateful to Lawrence Ein, Sijong Kwak, Robert Lazarsfeld, and Wenbo Niu, and he benefited from interesting discussions about \texttt{Macaulay2} computations in Example \ref{ex:Bettitable} with Kiryong Chung and Jong In Han.

%%%%%%%%%%%%%%%%%%%%%%%%%%%%%%%%%%%%%%%%%%%%%%%%%%
\section{Preliminaries}\label{sec:prelim}
%%%%%%%%%%%%%%%%%%%%%%%%%%%%%%%%%%%%%%%%%%%%%%%%%%

%%%%%%%%%%%%%%%%%%%%%%%%%%%%%%%%%%%%%%%%%%%%%%%%%%
\subsection{Syzygies}
%%%%%%%%%%%%%%%%%%%%%%%%%%%%%%%%%%%%%%%%%%%%%%%%%%
Let $X$ be a projective scheme, $B$ be a coherent sheaf on $X$, and $L$ be a very ample line bundle on $X$. The \emph{Koszul cohomology} $K_{p,q}(X, B; L)$ is the cohomology of the Koszul-type complex
\begin{footnotesize}
$$
\begin{array}{l}
\wedge^{p+1} H^0(X, L) \otimes H^0(X, B \otimes L^{q-1}) \longrightarrow \wedge^p  H^0(X, L)\otimes H^0(X, B \otimes L^q)
\longrightarrow \wedge^{p-1}H^0(X, L) \otimes H^0(X, B \otimes L^{q+1}).
\end{array}
$$
\end{footnotesize}\\[-15pt]
When $B=\sO_X$, we put $K_{p,q}(X, L) :=K_{p,q}(X, \sO_X; L)$ and $\kappa_{p,q}(X, L):=\dim_{\kk} K_{p,q}(X, L)$. Let $S:= \bigoplus_{m \geq 0} S^mH^0(X, L)$, and view $R=R(X,B;L):=\bigoplus_{m \in \mathbb{Z}} H^0(X, B \otimes L^m)$ as a graded $S$-module. If 
$$
 \xymatrix{
 0 & R \ar[l]& E_0 \ar[l]  & E_1  \ar[l] & \ar[l] ~\cdots~ & \ar[l] E_r \ar[l]  &  \ar[l]0 
 }
$$
is a minimal free resolution of $R$, then $E_p = \bigoplus_{q} K_{p,q}(X, B; L)  \otimes_{\kk} S(-p-q)$. We may think that $K_{p,q}(X, B; L)$ is the space of $p$-th syzygies of weight $q$.  For a globally generated vector bundle $E$ on $X$, we denote by $M_E$ the kernel of the evaluation map $H^0(X, E) \otimes \sO_X \to E$. The following is well-known.

\begin{proposition}[{cf. \cite[Proposition 2.1]{Park}}]\label{prop:koszulcoh}
Assume that $H^i(X, B \otimes L^m)=0$ for $i >0$ and $m >0$. For $q \geq 2$, we have $K_{p,q}(X, B; L)=H^{q-1}(X, \wedge^{p+q-1}M_L \otimes B \otimes L)$. If furthermore $H^{q-1}(X, B) = H^q(X, B) = 0$, then $K_{p,q}(X, B; L)=H^{q}(X, \wedge^{p+q}M_L \otimes B)$.
\end{proposition}

%%%%%%%%%%%%%%%%%%%%%%%%%%%%%%%%%%%%%%%%%%%%%%%%%%
\subsection{Castelnuovo--Mumford regularity}
%%%%%%%%%%%%%%%%%%%%%%%%%%%%%%%%%%%%%%%%%%%%%%%%%%
A coherent sheaf $\sF$ on $\nP^n$ is said to be \emph{$m$-regular} if $H^i(\nP^n, \sF(m-i))=0$ for $i>0$. By Mumford's theorem \cite[Theorem 1.8.3]{Lazarsfeld}, if $\sF$ is $m$-regular, then $\sF$ is $(m+1)$-regular.  The \emph{Castelnuovo--Mumford regularity} $\reg \sF$ is the minimum $m$ such that $\sF$ is $m$-regular. If $\sF$ fits into an exact sequence $\cdots \to \sF_2 \to \sF_1 \to \sF_0 \to \sF \to 0$ of coherent sheaves on $\nP^n$ and $\sF_i$ is $(m+i)$-regular for each $i \geq 0$, then $\sF$ is $m$-regular (\cite[Example 1.8.7]{Lazarsfeld}). If $\sF$ is $m$-regular and $E$ is an $m'$-regular vector bundle on $\nP^n$, then $\sF \otimes E$ is $(m+m')$-regular (\cite[Proposition 1.8.9]{Lazarsfeld}). We can think of the regularity of a coherent sheaf $\sG$ on a closed subscheme $X \subseteq \nP^n$ by regarding $\sG$ as a sheaf on $\nP^n$. 

%%%%%%%%%%%%%%%%%%%%%%%%%%%%%%%%%%%%%%%%%%%%%%%%%%
\subsection{Multilinear algebra}
%%%%%%%%%%%%%%%%%%%%%%%%%%%%%%%%%%%%%%%%%%%%%%%%%%
Let $V$ be a finite dimensional vector space over $\kk$. The symmetric group $\mf{S}_n$ naturally acts on $V^{\otimes n}$ by permuting the factors. The \emph{divided power} $D^n V$ is the subspace $\{ \omega \in V^{\otimes n} \mid \sigma(\omega) = \omega~\text{for all $\sigma \in \mf{S}_n$} \} \subseteq T^n V$, and the \emph{symmetric power} $S^n V$ is the quotient of $V^{\otimes n}$ by the span of $\sigma(\omega) - \omega$ for all $\omega \in  V^{\otimes n}$ and $\sigma \in \mf{S}_n$. We have a natural identification $D^n V = (S^n V^{\vee})^{\vee}$. By composing the inclusion of $D^n V$ into $V^{\otimes n}$ with the projection onto $S^n V$, we get a natural map $D^n V \to S^n V$. This is an isomorphism if $\operatorname{char}(\kk) =0$ or $\operatorname{char}(\kk) > n$, but it may be neither injective nor surjective in general. The \emph{wedge product} $\wedge^n V$ is the quotient of $V^{\otimes n}$ by the span of $v_1 \otimes \cdots \otimes v_n$ for all $v_1,\ldots,v_n \in V$ with $v_i \neq v_j$ for some $i \neq j$. We write $v_1 \wedge \cdots \wedge v_n$ for the class of $v_1 \otimes \cdots \otimes v_d$ in the quotient. There is a natural inclusion $\wedge^n V \to V^{\otimes n}$ given by $v_1 \wedge \cdots \wedge v_n \mapsto \sum_{\sigma \in \mf{S}_n} \sgn(\sigma) \sigma(v_1 \otimes \cdots \otimes v_n)$. This gives a splitting of the quotient map $V^{\otimes n} \to \wedge^n V$ if and only if $\operatorname{char}(\kk) =0$ or $\operatorname{char}(\kk) > n$. We refer to \cite[Section 3]{AFPRW1} for more details.

%%%%%%%%%%%%%%%%%%%%%%%%%%%%%%%%%%%%%%%%%%%%%%%%%%
\subsection{Projective spaces} 
%%%%%%%%%%%%%%%%%%%%%%%%%%%%%%%%%%%%%%%%%%%%%%%%%%
Throughout the paper, we put $U:=H^0(\nP^1, \sO_{\nP^1}(1))$, and fix a basis $1, x$ of $U$.
The monomials $1, x, \ldots, x^d$ form a basis of $S^d U$, and the divided power monomials $x^{(0)}, x^{(1)}, \ldots, x^{(d)}$ form a basis of $D^d U$. Let $1, y$ be the dual basis of $U^{\vee}$ to $1,x$. There is a natural indentification $U^{\vee} = \wedge^2 U \otimes U^{\vee} = U$ sending $1,y$ to $-x, 1$. Then $H^0(\nP^1, \sO_{\nP^1}(d)) = S^d U$ and $H^1(\nP^1, \sO_{\nP^1}(-d-2)) = D^d U$. Note that $M_{\sO_{\nP^1}(d)} = S^{d-1} U \otimes \sO_{\nP^1}(-1)$. 

\medskip

We may regard $\nP^n$ as the symmetric product of $\nP^1$. By permuting the components, $\mf{S}_{n}$ acts on the ordinary product $(\nP^1)^n$, and the line bundle $\sO_{\nP^1}(d)^{\boxtimes n}$ on $(\nP^1)^n$  descends to a line bundle $T_{n, \sO_{\nP^1}(d)} = \sO_{\nP^n}(d)$ on $\nP^n$ in such a way that $q_n^* T_{n, \sO_{\nP^1}(d)} = \sO_{\nP^1}(d)^{\boxtimes n}$, where $q_n \colon (\nP^1)^n \to \nP^n$ is the quotient map. Then $H^0(\nP^n, \sO_{\nP^n}(d)) = D^n S^d U$. Since $H^0(\nP^n, \sO_{\nP^n}(1)) = D^n U$, we get $H^0(\nP^n, \sO_{\nP^n}(d)) = S^d D^n U$. This gives \emph{Hermite reciprocity} $D^n S^d U = S^d D^n U$. 
If the action of $\mf{S}_{n}$ on $(\nP^1)^n$ is alternating, then $\sO_{\nP^1}(d)^{\boxtimes n}$ descends to $N_{n, \sO_{\nP^1}(d)} = \sO_{\nP^n}(d-n+1)$ and $H^0(\nP^n, \sO_{\nP^n}(d-n+1)) = \wedge^n S^d U$. This gives another \emph{Hermite reciprocity} 
\begin{equation}\label{eq:Hermite}
\wedge^n S^d U = S^{d-n+1} D^n U.
\end{equation}
See \cite[Remark 3.2]{AFPRW1} and \cite[Subsection 2.3]{Park}. Our Hermite reciprocity coincides with the explicit map constructed in \cite[Section 3]{AFPRW1} (see \cite{RS1}), but we will not use this fact. 

\medskip

Let $D_n$ be the image of the injective map $\nP^{n-1} \times \nP^1 \to \nP^{n} \times \nP^1$ given by $(\xi, z) \mapsto (\xi+z, z)$. Note that the effective divisor $D_n$ on $\nP^n \times \nP^1$ is defined by
$$
\sum_{i=0}^n (-1)^i (x^0 \wedge \cdots \wedge \widehat{x^i} \wedge \cdots \wedge x^n)  \otimes  x^i  \in  \wedge^n S^n U  \otimes S^n U = H^0(\nP^n \times \nP^1, \sO_{\nP^n}(1) \boxtimes \sO_{\nP^1}(n)).
$$
Consider the short exact sequence
\begin{equation}\label{eq:sesD_n}
\xymatrix{
0 \ar[r] & \sO_{\nP^n}(d-1) \boxtimes \sO_{\nP^1}(-n) \ar[r]^-{\cdot D_n} & \sO_{\nP^n}(d) \boxtimes \sO_{\nP^1} \ar[r] & \sO_{\nP^{n-1}}(d) \boxtimes \sO_{\nP^1}(d) \ar[r] & 0.
}
\end{equation}
Pushing forward to $\nP^1$ yields a short exact sequence
\begin{small}
$$
\xymatrix{
0 \ar[r] & \wedge^n M_{\sO_{\nP^1}(d+n-1)} \ar[r] & \wedge^n S^{d+n-1} U \otimes \sO_{\nP^1} \ar[r] &  \wedge^{n-1} M_{\sO_{\nP^1}(d+n-1)} \otimes \sO_{\nP^1}(d+n-1) \ar[r] & 0.
}
$$
\end{small}\\[-10pt]
On the other hand, $D_n$ can be also defined by
$$
\sum_{i=0}^n (-1)^i x^{(n-i)} \otimes x^i  \in  D^n U \otimes S^n U = H^0(\nP^n \times \nP^1, \sO_{\nP^n}(1) \boxtimes \sO_{\nP^1}(n)).
$$
Then we see that the map 
\begin{equation}\label{eq:D_n=multiplication}
\begin{gathered}
\xymatrixrowsep{0.18in}
\xymatrix{
H^1(\nP^n \times \nP^1,  \sO_{\nP^n}(d-1) \boxtimes \omega_{\nP^1}(-n)) \ar[r]^-{\cdot D_n} \ar@{=}[d] & H^1(\nP^n \times \nP^1,  \sO_{\nP^n}(d) \boxtimes  \omega_{\nP^1}) \ar@{=}[d]\\
 S^{d-1} D^n U \otimes D^n U  & S^d D^n U
}
\end{gathered}
\end{equation}
is given by $f \otimes x^{(i)} \mapsto (-1)^n fx^{(i)}$. We simply regard this as the multiplication map.

%%%%%%%%%%%%%%%%%%%%%%%%%%%%%%%%%%%%%%%%%%%%%%%%%%
\subsection{Secant varieties} 
%%%%%%%%%%%%%%%%%%%%%%%%%%%%%%%%%%%%%%%%%%%%%%%%%%
We recall the set-up of \cite{ENP1, ENP2}, and we present some preliminary results. Let $C$ be a smooth projective curve of genus $g \geq 0$, and $L$ be a  line bundle on $C$ with $\deg L \geq 2g+3$.\footnote{It is assumed that $\operatorname{char}(\kk) = 0$ in \cite{ENP1, ENP2}, but everything works when $g=0$ and $\operatorname{char}(\kk) \geq 0$ (see also \cite{RS1}).} We assume $\operatorname{char}(\kk) = 0$ whenever $g \geq 1$. We denote by $C^2=C \times C$ the ordinary product of $C$ and $C_2=C^2/\mf{S}_2$ the symmetric product of $C$. The quotient map $\sigma \colon C^2 \to C_2$ is given by $(x,y) \mapsto x+y$. If we regard $C_2$ as the Hilbert scheme of two points on $C$, then $\sigma$ is the universal family. For any line bundle $A$ on $C$, there is a line bundle $T_A$ on $C_2$ such that $\sigma^* T_A = A \boxtimes A$. Let $D$ be the diagonal of $C^2$, and $Q:=\sigma(D)$. Then $Q \cong C$ unless $g=0$ and $\operatorname{char}(\kk) = 2$. We may write $Q = 2\delta$ for some divisor $\delta$ on $C_2$. Note that $\sigma^* \delta = D$ and $\sigma_* \sO_{C^2} = \sO_{C_2}(-\delta) \oplus \sO_C$ (cf. \cite[Lemma 3.5]{ENP1}). We can write $K_{C_2} = T_{K_C}(-\delta)$. Consider the  \emph{tautological bundle} $E:=\sigma_*(\sO_C \boxtimes L)$ on $C_2$. We have $\rank E = 2$ and $\det E = T_L(-\delta)$. Let $B:=\nP(E)$, and $\pi \colon B \to C_2$ be the canonical fibration. As  $H^0(C_2, E) = H^0(C, L)$ and $E$ is globally generated, $|\sO_{\nP(E)}(1)|$ induces a map $\beta \colon B \to \nP^r=\nP H^0(C, L)$. Then $\Sigma:=\beta(B)$ is the \emph{secant variety} of $C$ in $\nP^r$, and $\beta \colon B \to \Sigma$ is  the blow-up of $\Sigma$ along $C$ (see \cite[Theorem 1.1]{ENP2}). Unless $g=0$ and $\deg L = 3$ (in this case $\Sigma = \nP^3$),  $\Sing \Sigma = C$  and $\beta \colon B \to \Sigma$ is a resolution of singularities. Let $Z:=\beta^{-1}(C) \cong C^2$. Then $\pi|_Z \colon Z \to C_2$ is just $\sigma$, and $\beta|_Z \colon Z = C \times C \to C$ is the second projection.

\begin{small}
$$
\xymatrix{
Z=C \times C \ar@{^{(}->}[r]  \ar[rd]_-{\sigma} & B=\nP(E) \ar[r]^-{\beta}  \ar[d]^-{\pi} & \Sigma \subseteq \nP^r \\
& C_2 &
}
$$
\end{small}

\begin{theorem}[{\cite[Theorems 1.1 and 1.2]{ENP1}}]\label{thm:secant}
$\Sigma \subseteq \nP^g$ is arithmetically Cohen--Macaulay. If $g=0$, then $\Sigma$ has rational singularities and $\reg \sO_{\Sigma} = 2$. If $g \geq 1$ and $\operatorname{char}(\kk) = 0$, then $\Sigma$ has normal Du Bois singularities and $\reg \sO_{\Sigma} = 4$.
\end{theorem}

Pick $H \in |\sO_{\nP(E)}(1)|$. We may write $K_B = -2H + \pi^* T_{K_C + L}(-2\delta)$ and $Z=2H - \pi^* T_L(-2\delta)$. Take $\overline{S} \in |2\delta|$.  Let $\widetilde{S}:=\pi^{-1}(\overline{S}) = \nP(E|_{\overline{S}})$, and $S:=\beta(\widetilde{S})$. We are mostly interested in the case $\overline{S}=Q$. In fact, $Q$ is a unique member in $|2\delta|$ when $g \geq 2$. Note that $\dim |2\delta| = 5$ when $g=0$ and $\dim |2\delta| =1$ when $g=1$. Assume that $\operatorname{char}(\kk) \neq 2$ when $g=0$. Note that $M_E|_Q = N_{C|\nP^r}^{\vee} \otimes L$ and $E|_Q = \mathscr{P}^1(L)$ is the first jet bundle. By \cite[Corollary 1.8]{Kaji}, $\mathscr{P}^1(L)$ is the unique nontrivial extension of $L$ by $\omega_C \otimes L$ when $\operatorname{char}(\kk) \nmid \deg L$. Let $\widetilde{T}:=\pi^{-1}(Q)$. Then $T:=\beta(\widetilde{T})$ is the \emph{tangent developable surface} of $C$ in $\nP^r$, and $\Sing T = C$.
Note that $\nu:=\beta|_{\widetilde{T}} \colon \widetilde{T} \to T$ is a resolution of singularities. 
 
 \begin{proposition}
$\deg S =2 \deg L + 2g-2$.
\end{proposition}

\begin{proof}
We have $\deg S = (H|_{\widetilde{S}})^2 = (\det E) \cdot \overline{S} = 2\deg L + 2g-2$.
\end{proof}

Recall from \cite[Theorem 5.2]{ENP1} that $\beta_*\sO_B(-Z) = \sI_{C|\Sigma}$ and $R^1 \beta_*\sO_B(-Z) = 0$.

\begin{lemma}\label{lem:R1O_B(-S-Z)}
$\beta_* \sO_B(-\widetilde{S}-Z) = \sI_{S|\Sigma}$ and $R^1 \beta_* \sO_B(-\widetilde{S}-Z) = 0$. 
\end{lemma}

\begin{proof}
As $\beta_* \sO_Z(-\widetilde{S}) = 0$, we have $\beta_* \sO_B(-\widetilde{S}-Z) = \beta_* \sO_B(-\widetilde{S}) = \sI_{S|\Sigma}$. For the second assertion, following \cite[Proof of Theorem 5.2 (2)]{ENP1}, we show that $R^1 \beta_* \sO_B(-\widetilde{S}-Z)_x = 0$ for any  $x \in C \subseteq \Sigma$. Let $F:=\beta^{-1}(x) \cong C$. By the formal function theorem, it suffices to prove that
$$
H^1(F, \sO_B(-\widetilde{S}-Z) \otimes \sO_B/\sI_{F|B}^{m}) = 0~~\text{ for $m \geq 1$}.
$$
It is enough to check that
$$
H^1(F, \sO_B(-\widetilde{S}-Z) \otimes \sI_{F|B}^{m}/\sI_{F|B}^{m+1}) = 0~~\text{ for $m \geq 0$}.
$$
As $(-\widetilde{S}-Z)|_Z = (L \boxtimes -L)(-4D)$ and $F$ is a fiber of the second projection $Z=C \times C \to C$, we have $\sO_B(-\widetilde{S}-Z)|_F = L(-4x)$. Note that $\sI_{F|B}^{m}/\sI_{F|B}^{m+1} = S^{m} N_{F|B}^{\vee}$. Recall from \cite[Proposition 3.13]{ENP1} that $N_{F|B}^{\vee} = \sO_C \oplus L(-2x)$. The problem is then reduced to verifying that
$$
H^1\big(C, (m+1)L +(-4-2m)x \big) = 0~~\text{ for $m \geq 0$}.
$$
This vanishing holds since $\deg \big( (m+1)L +(-4-2m)x \big) \geq (m+1)(2g+3) - 4-2m \geq 2g-1$. 
\end{proof}

Consider the commutative diagram with short exact sequences
\begin{small}
\begin{equation}\label{eq:commdiagonB}
\begin{gathered}
\xymatrixrowsep{0.25in}
\xymatrixcolsep{0.25in}
\xymatrix{
& & 0 \ar[d] & 0 \ar[d] & \\
0 \ar[r] & \sO_B(-\widetilde{S}-Z) \ar[r]^-{\cdot Z} \ar@{=}[d] & \sO_B(-\widetilde{S}) \ar[r] \ar[d]^-{\cdot  \widetilde{S}} & \sO_{Z}(-2D) \ar[r] \ar[d] & 0 \\
0 \ar[r] & \sO_B(-\widetilde{S}-Z) \ar[r]_-{\cdot \widetilde{S} + Z} & \sO_B \ar[r] \ar[d] & \sO_{\widetilde{S}+Z} \ar[r] \ar[d] & 0\\
& & \sO_{\widetilde{S}} \ar@{=}[r] \ar[d]  & \sO_{\widetilde{S}} \ar[d] \\
&  & 0 & ~0. &
}
\end{gathered}
\end{equation}
\end{small}\\[-5pt]
By Lemma \ref{lem:R1O_B(-S-Z)}, applying $\beta_*$ to the second row exact sequence in (\ref{eq:commdiagonB}), we find $\beta_* \sO_{\widetilde{S} + Z}  = \sO_S$ and $R^1 \beta_*  \sO_{\widetilde{S} + Z}  = R^1 \beta_* \sO_B  = H^1(C, \sO_C) \otimes \sO_C$, and we get a short exact sequence
\begin{equation}\label{eq:sesOSgeneral1}
\xymatrix{
0 \ar[r] & \beta_* \sO_B(-\widetilde{S}-Z) \ar[r] &  \sO_{\Sigma} \ar[r] &\sO_S \ar[r] & 0.
}
\end{equation}
Note that $R^1 \beta_*  \sO_{\widetilde{S} + Z}  = R^1 \beta_* \sO_B   = R^1 (\beta|_Z)_* \sO_Z =R^1 (\beta|_Z)_* \sO_Z(-D)$. The kernel of the map $R^1 (\beta|_Z)_* \sO_Z(-2D) \to R^1 (\beta|_Z)_* \sO_Z(-D)$ is $\sO_D(-D|_D) = \omega_C$. By applying $\beta_*$ to the right-most vertical exact sequence in (\ref{eq:commdiagonB}), we get a short exact sequence
\begin{equation}\label{eq:sesOSgeneral2}
\xymatrix{
0 \ar[r] & \sO_S \ar[r] & \beta_* \sO_{\widetilde{S}} \ar[r] & \omega_C \ar[r] & 0.
}
\end{equation}

\begin{proposition}\label{prop:omegaS}
The dualizing sheaf $\omega_S$ is trivial.
\end{proposition}

\begin{proof}
Consider two short exact sequences
\begin{small}
\begin{equation}\label{eq:twoses}
\xymatrixcolsep{0.28in}
\xymatrix{
0 \ar[r] & \omega_B \ar[r] & \omega_B(Z)\ar[r]  & \omega_Z \ar[r] & 0~\text{ }~\text{ }~\text{ and }~\text{ }~~\text{ } 0 \ar[r] & \omega_B \ar[r] & \omega_B(\widetilde{S}) \ar[r] & \omega_{\widetilde{S}} \ar[r] & 0.
}
\end{equation}
\end{small}\\[-8pt]
By Theorem \ref{thm:secant} for $g=0$ and Grauert--Riemenschneider vanishing for $g \geq 1$, we have $R^1 \beta_* \omega_B = 0$. Then  $R^1 \beta_* \omega_B(Z) =R^1 \beta_* \omega_Z =  \omega_C$ and $R^1 \beta_* \omega_B(\widetilde{S}) = 0$. By taking $-\otimes \sO_B(\widetilde{S})$ to the left of (\ref{eq:twoses}), we see that $R^1 \beta_* \omega_B(\widetilde{S}+Z) = 0$. When $g=0$, we have $\beta_* \omega_B(Z) = \beta_* \omega_B = \omega_{\Sigma}$ by Theorem \ref{thm:secant}. When $g \geq 1$, Theorem \ref{thm:secant} and \cite[Theorem 1.1]{KSS} show that $\beta_* \omega_B(Z) = \omega_{\Sigma}$. As $\omega_{\widetilde{S}}(Z|_{\widetilde{S}}) = \sO_{\widetilde{S}}$, applying $\beta_*$ to the short exact sequence with the consideration of (\ref{eq:sesOSgeneral2}),
$$
\xymatrix{
0 \ar[r] & \omega_B(Z) \ar[r]^-{\cdot \widetilde{S}} & \omega_B(\widetilde{S}+Z) \ar[r] & \sO_{\widetilde{S}} \ar[r] & 0,
}
$$
we obtain a short exact sequence
$$
\xymatrix{
0 \ar[r] & \omega_{\Sigma} \ar[r] & \omega_{\Sigma}(S) \ar[r] & \sO_S \ar[r] & 0 .
}
$$
Since $\Sigma$ is Cohen--Macaulay, we conclude that $\omega_S = \sO_S$.
\end{proof}

When $g=0$ and $S=T$ is the tangent developable surface, the above proposition was shown in \cite[Corollary 5.16]{AFPRW1} 

%%%%%%%%%%%%%%%%%%%%%%%%%%%%%%%%%%%%%%%%%%%%%%%%%%
\subsection{Rational curve case} 
%%%%%%%%%%%%%%%%%%%%%%%%%%%%%%%%%%%%%%%%%%%%%%%%%%
Assume that $\operatorname{char}(\kk) \neq 2$. Let $C \subseteq \nP^g$ be a rational normal curve of degree $g \geq 3$, and $L:=\sO_{\nP^1}(g)$. When we consider a rational normal curve, $g$ is not the genus  but the degree. Note that $C_2 = \nP^2$ and $\sO_{\nP^2}(\delta) = \sO_{\nP^2}(1)$. We have $T_{\sO_{\nP^1}(d)} =  \sO_{\nP^2}(d)$, and put $N_{\sO_{\nP^1}(d)}:=T_{\sO_{\nP^1}(d)}(-\delta) = \sO_{\nP^2}(d-1)$. Then $\sigma^* T_{\sO_{\nP^1}(d)} = \sO_{\nP^1}(d) \boxtimes \sO_{\nP^1}(d)$ and $\sigma^* N_{\sO_{\nP^1}(d)} = \sO_{\nP^1}(d-1) \boxtimes \sO_{\nP^1}(d-1)$. Note that $\pi_* \sO_B(-Z) = 0$ and $R^1 \pi_* \sO_B(-Z) = \sO_{\nP^2}(-\delta)$. By applying $\pi_*$ to the short exact sequence
\begin{equation}\label{eq:sesZonB}
\xymatrix{
0 \ar[r] & \sO_B(-Z) \ar[r]^-{\cdot Z}  & \sO_B \ar[r] & \sO_Z \ar[r] & 0,
}
\end{equation}
we get a splitting short exact sequence
$$
\xymatrix{
0 \ar[r] & \sO_{\nP^2} \ar[r] & \sigma_* \sO_Z \ar[r] & \sO_{\nP^2}(-\delta) \ar[r] & 0.
}
$$
Then we obtain a canonically splitting short exact sequence
\begin{small}
\begin{equation}\label{eq:sesZgivessplitting}
\begin{gathered}
\xymatrixrowsep{0.18in}
\xymatrix{
0 \ar[r] & H^0(\nP^2, T_{\sO_{\nP^1}(d) }) \ar[r] \ar@{=}[d] & H^0(\nP^1 \times \nP^1,  \sO_{\nP^1}(d) \boxtimes \sO_{\nP^1}(d)) \ar[r] \ar@{=}[d] & H^0(\nP^2, N_{\sO_{\nP^1}(d) }) \ar[r] \ar@{=}[d] & 0\\
& D^2 S^d U & S^d U \otimes S^d U & \wedge^2 S^d U. &
}
\end{gathered}
\end{equation}
\end{small}

\begin{lemma}\label{lem:M_E}
$M_E = S^{g-2} U \otimes \sO_{\nP^2}(-1)$.
\end{lemma}

\begin{proof}
Recall that there is an injective map $D_2 = \nP^1 \times \nP^1 \to \nP^2 \times \nP^1$. Let $p_1 \colon \nP^2 \times \nP^1 \to \nP^2$ be the first projection. Then $(p_1)|_{D_2} = \sigma$. By applying $p_{1,*}$ to the short exact sequence
$$
\xymatrix{
0 \ar[r] & \sO_{\nP^2}(-1) \boxtimes \sO_{\nP^1}(g-2) \ar[r]^-{\cdot D_2} & \sO_{\nP^2} \boxtimes \sO_{\nP^1}(g) \ar[r] & \sO_{\nP^1} \boxtimes \sO_{\nP^1}(g) \ar[r] & 0,
}
$$
we get a short exact sequence
$$
\xymatrix{
0 \ar[r] & S^{g-2} U \otimes \sO_{\nP^2}(-1) \ar[r] & S^g U \otimes \sO_{\nP^2} \ar[r] & E \ar[r] & 0,
}
$$
where $H^0(\nP^2, E) \otimes \sO_{\nP^2} = S^g U \otimes \sO_{\nP^2} \to E$ is the evaluation map. Thus the lemma follows.
\end{proof}

\begin{proposition}\label{prop:ACM}
$(1)$ $K_B+\widetilde{S}+Z = 0$.\\
$(2)$ $S \subseteq \nP^g$ is arithmetically Cohen--Macaulay and $\reg \sO_S  =3$.\\
$(3)$ The Hilbert function of $S \subseteq \nP^g$ is given by $H_S(t)=(g-1)t^2 + 2$ for $t \geq 1$.
\end{proposition}

\begin{proof}
The assertion $(1)$ follows from the direct computation:
$$
K_B+ \widetilde{S} +Z = (-2H + \pi^*\sO_{\nP^2}(g-4) ) + \pi^*\sO_{\nP^2}(2) +(2H - \pi^*\sO_{\nP^2}(g-2)) = 0.
$$ 
As $\beta_*\sO_B(-\widetilde{S}-Z)=\beta_* \omega_B = \omega_{\Sigma}$ by Theorem \ref{thm:secant}, the short exact sequence (\ref{eq:sesOSgeneral1}) is
\begin{equation}\label{eq:sessecant/tangent}
\xymatrix{
0 \ar[r] & \omega_{\Sigma} \ar[r] & \sO_{\Sigma} \ar[r] & \sO_S \ar[r] & 0.
}
\end{equation}
By Theorem \ref{thm:secant}, we readily obtain $(2)$. Now, $H_T(t)$ agrees with the Hilbert polynomial $P_S(t)$ for $t \geq 1$. Note that $\det P_S(t)=2$ and the leading coefficient of $P_T(t)$ is $(\deg S)/2=g-1$. As $P_S(0)=\chi(\sO_S)=2$ and $P_S(1)=g+1$, we get $(3)$.
\end{proof}

When $S=T$ is the tangent developable surface, Proposition \ref{prop:ACM} $(2)$ was first shown by Schreyer \cite[Proposition 6.1]{Schreyer1} (see also \cite[Theorem 5.1]{AFPRW1}). Observe that $S \subseteq \nP^g$ has the same Hilbert polynomial with a K3 surface of degree $2g-2$ in $\nP^g$.

\medskip

Suppose that $\overline{S}=Q$. Here $Q=\sigma(D)$ is a smooth conic in $\nP^2$. When $\operatorname{char}(\kk) = 0$ or $\operatorname{char}(\kk) \geq (g+2)/2$, \cite[Corollary 1.8]{Kaji} implies that $E|_Q = \sO_{\nP^1}(g-1) \oplus  \sO_{\nP^1}(g-1)$. In this case, $\widetilde{T} = \nP^1 \times \nP^1$, and $\sO_{\widetilde{T}}(H|_{\widetilde{T}}) = \sO_{\nP^1}(g-1) \boxtimes \sO_{\nP^1}(1)$.  
Consider the short exact sequence
\begin{small}
\begin{equation}\label{eq:sesWahlmap}
\begin{gathered}
\xymatrixrowsep{0.18in}
\xymatrix{
0 \ar[r] & H^0(\nP^2, \sO_{\nP^2}(d-1)) \ar[r]^-{\cdot Q} \ar@{=}[d] & H^0(\nP^2, \sO_{\nP^2}(d+1)) \ar[r] \ar@{=}[d] & H^0(Q, \sO_{Q}(d+1)) \ar[r] \ar@{=}[d] & 0 \\
& \wedge^2 S^d U & \wedge^2 S^{d+2} U & S^{2d+2} U.
}
\end{gathered}
\end{equation}
\end{small}\\[-5pt]
Since $\sigma^*Q = 2D$ and the diagonal $D$ of $\nP^1 \times \nP^1$ is defined by 
$$
x \otimes 1 - 1 \otimes x \in U \otimes U =  H^0(\nP^1 \times \nP^1,  \sO_{\nP^1}(1) \boxtimes \sO_{\nP^1}(1)),
$$
the inclusion $H^0(\nP^2, \sO_{\nP^2}(d-1)) \to H^0(\nP^2, \sO_{\nP^2}(d+1))$ in (\ref{eq:sesWahlmap}) is
$$
\wedge^2 S^d U \longrightarrow \wedge^2 S^{d+2} U,~~x^i \wedge x^j \longmapsto x^{i+2} \wedge x^j - 2 x^{i+1} \wedge x^{j+1} + x^{i} \wedge x^{j+2}.
$$
The surjection $H^0(\nP^2, \sO_{\nP^2}(d+1)) \to H^0(Q, \sO_{Q}(d+1))$ in (\ref{eq:sesWahlmap}) is
\begin{equation}\label{eq:wahlmap}
\mu_{d+2} \colon \wedge^2 S^{d+2} U \longrightarrow  S^{2d+2} U,~~x^i \wedge x^j \longmapsto (i-j) x^{i+j-1},
\end{equation}
which is the \emph{Wahl map} (or the \emph{Gaussian map}). See \cite{BEL, Wahl} for details on Wahl maps.

\medskip

Suppose that $\overline{S} = 2\ell$, where $\ell \subseteq \nP^2$ is a line meeting $Q$ at two distinct points. We may write $E|_{\ell} = \sO_{\nP^1}(a) \oplus \sO_{\nP^1}(b)$ for some integers $b \geq a \geq 1$.
As $\det E = N_L = \sO_{\nP^2}(g-1)$, we have $a+b=g-1$. Note that $\sigma^{-1}(\ell) = \nP^1$ and $\sigma |_{\sigma^{-1}(\ell)} \colon \sigma^{-1}(\ell)  \to \ell$ is a 2-to-1 map. By restricting the following short exact sequence
$$
\xymatrix{
0 \ar[r] & \sO_{\nP^1}(g-1) \boxtimes \sO_{\nP^1}(-1) \ar[r] & \sigma^* E \ar[r] & \sO_{\nP^1} \boxtimes \sO_{\nP^1}(g) \ar[r] & 0
}
$$
to $\sigma^{-1}(\ell)$, we get a short exact sequence
$$
\xymatrix{
0 \ar[r] & \sO_{\nP^1}(g-2)  \ar[r] & \sO_{\nP^1}(2a) \oplus \sO_{\nP^1}(2b) \ar[r] & \sO_{\nP^1}(g) \ar[r] & 0.
}
$$
It follows that $(a,b) = (\lfloor (g-1)/2 \rfloor, \lfloor g/2 \rfloor)$. Then $S$ is a double structure on a rational normal surface scroll $S(a,b)=\nP(E|_{\ell}) \subseteq \nP^g$. Recall that $H^1(S, \sO_S)=0$ and $\omega_S = \sO_S$. Thus $S=X(\lfloor (g-1)/2 \rfloor, \lfloor g/2 \rfloor) \subseteq \nP^g$ is a \emph{K3 carpet} (see \cite[Definition 1.2 and Theorem 1.3]{GP}\footnote{The proof of \cite[Theorem 1.3]{GP} also works in positive characteristic.}). On the other hand, consider the short exact sequence
\begin{small}
\begin{equation}\label{eq:ses2ell}
\begin{gathered}
\xymatrixrowsep{0.18in}
\xymatrix{
0 \ar[r] & H^0(\nP^2, \sO_{\nP^2}(d-1)) \ar[r]^-{\cdot 2\ell} \ar@{=}[d] & H^0(\nP^2, \sO_{\nP^2}(d+1)) \ar[r] \ar@{=}[d] & H^0(2\ell, \sO_{2\ell}(d+1)) \ar[r] & 0\\
& \wedge^2 S^d U & \wedge^2 S^{d+2} U. & 
}
\end{gathered}
\end{equation}
\end{small}\\[-5pt]
We may assume that $\sigma^{-1}(\ell)$ in $\nP^1 \times \nP^1$ is defined by
$$
x \otimes 1 + 1 \otimes x \in U \otimes U = H^0(\nP^1 \times \nP^1, \sO_{\nP^1}(1) \boxtimes \sO_{\nP^1}(1)).
$$
Then the inclusion $H^0(\nP^2, \sO_{\nP^2}(d-1)) \to H^0(\nP^2, \sO_{\nP^2}(d+1))$ in (\ref{eq:ses2ell}) is
$$
\wedge^2 S^d U \longrightarrow \wedge^2 S^{d+2} U,~~x^i \wedge x^j \longmapsto x^{i+2} \wedge x^j + 2 x^{i+1} \wedge x^{j+1} + x^{i} \wedge x^{j+2}.
$$
It is a direct summand of the inclusion
$$
H^0(\nP^1 \times \nP^1, \sO_{\nP^1}(d) \boxtimes \sO_{\nP^1}(d)) \xrightarrow{~\cdot 2\sigma^{-1}(\ell)~}  H^0(\nP^1 \times \nP^1, \sO_{\nP^1}(d+2) \boxtimes \sO_{\nP^1}(d+2))
$$
whose cokernel is $S^{2d+2}U \oplus S^{2d+4}U$. The surjection $H^0(\nP^2, \sO_{\nP^2}(d+1)) \to H^0(2\ell, \sO_{2\ell}(d+1))$ in (\ref{eq:ses2ell}) factors through
\begin{small}
\begin{equation}\label{eq:tau}
\begin{gathered}
\tau_{d+2} \colon \wedge^2 S^{d+2} U \longrightarrow  \begin{array}{c} S^{2d+2} U \\[-4pt] \oplus \\S^{2d+4} U \end{array},~~x^i \wedge x^j \longmapsto \begin{cases} \big( (-1)^i(i-j)x^{i+j-1}, 0 \big) & \text{if $i \equiv j ~(\operatorname{mod} 2)$} \\ \big(0, (-1)^i x^{i+j} \big) & \text{if $i \not\equiv j~ (\operatorname{mod} 2)$ } \end{cases}
\end{gathered}
\end{equation}
\end{small}\\[-5pt]
in such a way that $\im(\tau_{d+2})$ injects into $H^0(2\ell, \sO_{2\ell}(d+1))$.

%%%%%%%%%%%%%%%%%%%%%%%%%%%%%%%%%%%%%%%%%%%%%%%%%%
\section{Proofs of Theorems \ref{thm:AFPRW} and \ref{thm:K3carpet}}\label{sec:proof1}
%%%%%%%%%%%%%%%%%%%%%%%%%%%%%%%%%%%%%%%%%%%%%%%%%%

Assume that $\operatorname{char}(\kk) \neq 2$. Let $C \subseteq \nP^g$ be a rational normal curve of degree $g \geq 3$, and $L:=\sO_{\nP^1}(g)$. We use the notations in Section \ref{sec:prelim}. From (\ref{eq:sessecant/tangent}), we see that
$$
K_{p,2}(S, \sO_S(1)) = \coker\left(K_{p,2}(\Sigma, K_{\Sigma}; \sO_{\Sigma}(1)) \xrightarrow{~\rho_{p+2}~} K_{p,2}(\Sigma, \sO_{\Sigma}(1))\right).
$$
Thus $K_{p,2}(S, \sO_S(1)) = 0$ if and only if $\rho_{p+2}$ is surjective. Recall from Theorem \ref{thm:secant} that $\Sigma$ has rational singularities. As $\beta^* M_{\sO_{\Sigma}(1)} = M_H$, by Proposition \ref{prop:koszulcoh}, we find
$$
\begin{array}{l}
K_{p,2}(\Sigma, K_{\Sigma}; \sO_{\Sigma}(1)) = H^2(\Sigma, \wedge^{p+2} M_{\sO_{\Sigma}(1)} \otimes \omega_{\Sigma}) = H^2(B, \wedge^{p+2} M_H \otimes \omega_B);\\[3pt]
 K_{p,2}(\Sigma, \sO_{\Sigma}(1)) = H^2(\Sigma, \wedge^{p+2} M_{\sO_{\Sigma}(1)}) = H^2(B, \wedge^{p+2} M_H).
\end{array}
$$
In (\ref{eq:commdiagonB}), we have $-\widetilde{S}-Z = K_B, ~-\widetilde{S} = K_B+Z,~-2D= K_Z$. Then $\rho_{p+2}$ fits into the following commutative diagram with exact sequences induced from (\ref{eq:commdiagonB}):
\begin{footnotesize}
\begin{equation}\label{eq:coWahl+Koszul}
\begin{gathered}
\xymatrixrowsep{0.25in}
\xymatrixcolsep{0.5in}
\xymatrix{
& H^1(\widetilde{S}, \wedge^{p+2} M_H|_{\widetilde{S}}) \ar@{^{(}->}[d]_-{\alpha_{p+2}}  \ar[rd]^-{\gamma_{p+2}}  \\
H^2(B, \wedge^{p+2} M_H \otimes \omega_B) \ar@{^{(}->}[r]   \ar[rd]_-{\rho_{p+2}} & H^2(B, \wedge^{p+2} M_H \otimes \omega_B(Z)) \ar[r]_-{\delta_{p+2}} \ar@{->>}[d] & H^2(Z, \omega_{\nP^1} \boxtimes \wedge^{p+2} M_L \otimes \omega_{\nP^1})\\
 & H^2(B, \wedge^{p+2} M_H). &
}
\end{gathered}
\end{equation}
\end{footnotesize}\\[-7pt]
It is easy to check that $\rho_{p+2}$ is surjective if and only if $\gamma_{p+2} = \delta_{p+2} \circ \alpha_{p+2}$ surjects onto $\im(\delta_{p+2})$. We shall prove the latter for $0 \leq p \leq \lfloor (g-3)/2 \rfloor$ when $S=T$ is the tangent developable surface and $S=X(\lfloor (g-1)/2 \rfloor, \lfloor g/2 \rfloor)$ is a K3 carpet. To this end, we first describe $\gamma_{p+2}=\delta_{p+2} \circ \alpha_{p+2}$. 
Let $q:=g-p-3,~V:=D^{p+2}U,~W:= H^1(\overline{S}, \sO_{\overline{S}}(-p-2))$. 

\begin{lemma}\label{lem:gamma_{p_2}}
The map $\gamma_{p+2}$ is the composition
$$
S^q V \otimes W \xrightarrow{~\alpha_{p+2}~} S^q V \otimes \wedge^2 V \xrightarrow{~\delta_{p+2}~} S^{q+1} V \otimes V
$$
such that $\delta_{p+2}$ is the Koszul differential given by $f \otimes (x^{(i)} \wedge x^{(j)}) \longmapsto f x^{(i)} \otimes x^{(j)} - f x^{(j)} \otimes x^{(i)}$ and $\alpha_{p+2}=\id_{S^q V  \otimes \Delta_{p+2}}$, where $\Delta_{p+2}$ fits into the following short exact sequence:\footnote{When $S=T$, one can easily confirm that the map $\gamma$ in the introduction coincides with the map $\gamma_{p+2}$ by replacing the range by $\im(\delta_{p+2}) = K_{p,1}(\nP^1, \omega_{\nP^1}; \sO_{\nP^1}(g))$. In this case, $\Delta_{p+2}$ can be thought as the dual of the Wahl map $\mu_{p+2}$.} 
\begin{small}
$$
\xymatrix{
0 \ar[r] & H^1(\overline{S}, \sO_{\overline{S}}(-p-2)) \ar[r]^-{\Delta_{p+2}} & H^2(\nP^2, \sO_{\nP^2}(-p-2)(-\overline{S})) \ar[r]^-{\cdot \overline{S}} & H^2(\nP^2, \sO_{\nP^2}(-p-2)) \ar[r] & 0.
}
$$
\end{small}
\end{lemma}

\begin{proof}
First, we give a description of $\alpha_{p+2}$.
As $\pi_* M_H = M_E$, we have a short exact sequence
\begin{equation}\label{eq:sesM_EM_H}
\xymatrix{
0 \ar[r] & \pi^* M_E \ar[r] & M_H \ar[r] & \sO_B(-H) \otimes \pi^* N_L \ar[r] & 0.
}
\end{equation}
Then we get $\pi_* \wedge^{p+2} M_H = \wedge^{p+2} M_E$ and $R^1 \pi_* \wedge^{p+2} M_H = 0$. By restricting (\ref{eq:sesM_EM_H}) to $\widetilde{S}$, we also get $(\pi|_{\widetilde{S}})_* \wedge^{p+2} M_H|_{\widetilde{S}} = \wedge^{p+2} M_E|_{\overline{S}}$ and $R^1 (\pi|_{\widetilde{S}})_* \wedge^{p+2} M_H|_{\widetilde{S}} = 0$. Recall from Lemma \ref{lem:M_E} that $M_E=  S^{g-2} U \otimes \sO_{\nP^2}(-1)$. Thus we find
\begin{small}
$$
\begin{array}{l}
H^1(\widetilde{S}, \wedge^{p+2} M_H|_{\widetilde{S}}) = H^1(\overline{S}, \wedge^{p+2} M_E|_{\overline{S}}) =\wedge^{p+2} S^{g-2} U \otimes H^1(\overline{S}, \sO_{\overline{S}}(-p-2));\\[3pt]
H^2(B, \wedge^{p+2} M_H \otimes \omega_B(Z)) = H^2(\nP^2, \wedge^{p+2} M_E \otimes \sO_{\nP^2}(-\overline{S})) =\wedge^{p+2} S^{g-2} U \otimes H^2(\nP^2, \sO_{\nP^2}(-p-2)(-\overline{S}));\\[3pt]
H^2(B, \wedge^{p+2} M_H) = H^2(\nP^2, \wedge^{p+2} M_E) = \wedge^{p+2} S^{g-2} U\otimes H^2(\nP^2, \sO_{\nP^2}(-p-2)).
\end{array}
$$
\end{small}\\[-6pt]
By considering the vertical short exact sequence in (\ref{eq:coWahl+Koszul}), we obtain $\alpha_{p+2} = \id_{\wedge^{p+2} S^{g-2} U} \otimes \Delta_{p+2}$. By Hermite reciprocity (\ref{eq:Hermite}), we have $\wedge^{p+2} S^{g-2} U = S^q V$.

\medskip

To describe $\delta_{p+2}$, we restrict (\ref{eq:sesM_EM_H}) to $Z$ to get a short exact sequence
\begin{small}
$$
\xymatrixrowsep{0.18in}
\xymatrixcolsep{0.14in}
\xymatrix{
0 \ar[r] &  \sigma^* M_E \ar[r] \ar@{=}[d]  & \sO_{\nP^1} \boxtimes M_L \ar[r] \ar@{=}[d] & ( \sO_{\nP^1} \boxtimes -L) \otimes \sigma^* N_L  \ar@{=}[d]  \ar[r] & 0\\
&   M_{\sO_{\nP^1}(g-1)}  \boxtimes  \sO_{\nP^1}(-1) & S^{g-1} U \otimes \sO_{\nP^1} \boxtimes  \sO_{\nP^1}(-1)   & \sO_{\nP^1}(g-1)  \boxtimes  \sO_{\nP^1}(-1), &
 }
$$
\end{small}\\[-6pt]
where the maps are identity on $\sO_{\nP^1}(-1)$ and the map $S^{g-1}U \otimes \sO_{\nP^1} \to \sO_{\nP^1}(g-1)$ is the evaluation map. The induced map $H^2(Z, \sigma^* \wedge^{p+2} M_E \otimes \omega_Z) \to H^2(Z,  \omega_{\nP^1} \boxtimes \wedge^{p+2} M_L \otimes   \omega_{\nP^1})$ can be written as $\delta_{p+2}' \otimes \id_{V}$, where
$$
\delta_{p+2}' \colon H^1(\nP^1, \wedge^{p+2} M_{\sO_{\nP^1}(g-1)} \otimes \omega_{\nP^1}) \longrightarrow  \wedge^{p+2} S^{g-1} U \otimes H^1(\nP^1, \omega_{\nP^1})
$$
is a map and $V=D^{p+2} U = H^1(\nP^1, \omega_{\nP^1}(-p-2))$. By considering the push-forward of (\ref{eq:sesD_n}) to $\nP^1$, we can identify $\delta_{p+2}'$ with
$$
H^1(\nP^{p+2} \times \nP^1, \sO_{\nP^{p+2}}(g-p-3) \boxtimes  \omega_{\nP^1}(-p-2)) \xrightarrow{~\cdot D_{p+2}~}   H^1( \nP^{p+2} \times \nP^1, \sO_{\nP^{p+2}}(g-p-2) \boxtimes \omega_{\nP^1}).
$$
This is the map (\ref{eq:D_n=multiplication}) with $n=p+2$ and $d=g-p-2=q+1$, so $\delta_{p+2}' \colon S^q V \otimes V \to S^{q+1} V$ is the multiplication map, where we use Hermite reciprocity (\ref{eq:Hermite}) to have that
$$
\begin{array}{l}
H^1(\nP^1, \wedge^{p+2} M_{\sO_{\nP^1}(g-1)} \otimes \omega_{\nP^1}) = \wedge^{p+2} S^{g-2} U \otimes D^{p+2} U = S^q V \otimes V;\\[3pt]
\wedge^{p+2} S^{g-1} U \otimes H^1(\nP^1, \omega_{\nP^1}) = \wedge^{p+2} S^{g-1} U = S^{q+1} V.
\end{array}
$$

\medskip

Now, consider a commutative diagram with short exact sequences
\begin{small}
$$
\xymatrixrowsep{0.18in}
\xymatrixcolsep{0.5in}
\xymatrix{
0 \ar[r] & \pi^* M_E \otimes \sO_B(-Z) \ar[r]^-{\cdot Z} \ar[d] & \pi^* M_E \ar[r] \ar[d] & \sigma^* M_E \ar[r] \ar[d] & 0 \\
0 \ar[r] & M_H \otimes \sO_B(-Z) \ar[r]_-{\cdot Z} & M_H \ar[r] & \sO_{\nP^1} \boxtimes M_L \ar[r] & 0.
}
$$
\end{small}\\[-5pt]
The top row is $(S^{g-2} U \otimes \pi^* \sO_{\nP^2}(-1)) \otimes$(\ref{eq:sesZonB}). The diagram induces a commutative diagram
\begin{small}
$$
\begin{gathered}
\xymatrixrowsep{0.18in}
\xymatrixcolsep{1.0in}
\xymatrix{
 H^2(B, \pi^* \wedge^{p+2} M_E \otimes \omega_B(Z)) \ar@{^{(}->}[r]  \ar@{=}[d] & H^2(Z, \sigma^* \wedge^{p+2} M_E \otimes \omega_Z) \ar[d]^-{\delta_{p+2}' \otimes \id_V}  \\
 H^2(B, \wedge^{p+2} M_H \otimes \omega_B(Z)) \ar[r]_-{\delta_{p+2}} & H^2(Z,  \omega_{\nP^1} \boxtimes \wedge^{p+2} M_L \otimes \omega_{\nP^1}).
 }
\end{gathered}
$$
\end{small}\\[-5pt]
Recalling $\omega_B(Z) = \pi^* \omega_{\nP^2}(1)$ and $\wedge^{p+2} S^{g-2} U  = S^q V$ and considering (\ref{eq:sesZonB}), we see that the upper horizontal injection can be written as $\id_{S^q V} \otimes \iota_V$, where 
$$
\iota_V \colon H^2(\nP^2, \omega_{\nP^2}(-p-1)) \longrightarrow H^2(Z, (\sO_{\nP^1}(-p-2) \boxtimes \sO_{\nP^1}(-p-2))\otimes \omega_Z)
$$
is the dual of the canonical surjection in (\ref{eq:sesZgivessplitting}) with $d=p+2$. In other words, the map $\iota_V \colon \wedge^2 V \to V \otimes V$ is given by $\iota_V (x^{(i)} \wedge x^{(j)}) =  x^{(i)} \otimes x^{(j)} - x^{(j)} \otimes x^{(i)}$. Hence we conclude that 
$\delta_{p+2} =(\delta_{p+2}' \otimes \id_V) \circ  (\id_{S^q V} \otimes \iota_V) \colon S^q V \otimes \wedge^2 V \to S^{q+1} V \otimes V$ 
is the Koszul differential.
\end{proof}

From now on, we regard $\gamma_{p+2} \colon W \otimes S^q V \to \im(\delta_{p+2})$ (the target $S^{q+1}V \otimes V$ is replaced by $\im(\delta_{p+2})$). Our aim is to show the surjectivity of $\gamma_{p+2}$ for $0 \leq p \leq \lfloor (g-3)/2 \rfloor$.\footnote{The Koszul complex $S^q V \otimes \wedge^2 V \to S^{q+1}V \otimes V \to S^{q+2} V$ is exact, so $\im(\delta_{p+2}) = \ker( S^{q+1}V \otimes V \to S^{q+2} V)$. Thus the surjectivity of $\gamma_{p+2}$ for $0 \leq p \leq \lfloor (g-3)/2 \rfloor$ is equivalent to (\ref{eq:vanKosmod}): $W_q(V, W) = 0$ for $q=g-p-3 \geq p \geq 0$.}
Viewing $V = H^0(\nP^{p+2}, \sO_{\nP^{p+2}}(1))$ and putting $M_V := M_{\sO_{\nP^{p+2}}(1)}$, we have a short exact sequence
$$
\xymatrix{
0 \ar[r] & \wedge^2 M_V \ar[r] & \wedge^2 V \otimes \sO_{\nP^{p+2}} \ar[r] & M_V (1) \ar[r] & 0.
}
$$
Consider the composition
\begin{equation}\label{eq:claim}
W \otimes \sO_{\nP^{p+2}} \xrightarrow{~\Delta_{p+2} \otimes \sO_{\nP^{p+2}}~} \wedge^2 V \otimes \sO_{\nP^{p+2}} \longrightarrow M_V (1).
\end{equation}

\begin{lemma}\label{lem:vankosmod}
Assume that $\operatorname{char}(\kk) = 0$ or $\operatorname{char}(\kk) \geq \lfloor (g-1)/2 \rfloor$. If the composition (\ref{eq:claim}) is surjective, then $\gamma_{p+2}$ is surjective for $0 \leq p \leq \lfloor (g-3)/2 \rfloor$.
\end{lemma}

\begin{proof}
We can form a short exact sequence
\begin{equation}\label{eq:sesKW}
\xymatrix{
0 \ar[r] & K \ar[r] & W \otimes \sO_{\nP^{p+2}} \ar[r]^-{\text{(\ref{eq:claim})}} & M_V(1)  \ar[r] & 0,
}
\end{equation}
where $K$ is a vector bundle with $\rank K = p+1$ and $\det K = \sO_{\nP^{p+2}}(-p-1)$. Notice that $\gamma_{p+2}$ is the $k=q$ case of the map
$$
\gamma_k' \colon W \otimes H^0(\nP^{p+2}, \sO_{\nP^{p+2}}(k)) \longrightarrow H^0(\nP^{p+2}, M_V (k+1)).
$$
As $q=g-p-3 \geq p$, it suffices to show the surjectivity of $\gamma_{k}'$ for $k \geq p$. To this end, consider the dual of (\ref{eq:sesKW}):
$$
\xymatrix{
0 \ar[r] & M_V^{\vee} (-1) \ar[r] & W^{\vee} \otimes \sO_{\nP^{p+2}} \ar[r] & K^{\vee} \ar[r] & 0.
}
$$
Since $M_V^{\vee}(-1)$ is $1$-regular and $W^{\vee} \otimes \sO_{\nP^{p+2}}$ is $0$-regular, it follows that $K^{\vee}$ is $0$-regular. Then $(K^{\vee})^{\otimes p}$ is $0$-regular, and so is $\wedge^{p} K^{\vee}$ because $\operatorname{char}(\kk) = 0$ or $\operatorname{char}(\kk) \geq \lfloor (g-1)/2 \rfloor \geq p+1$. Thus $K=\wedge^{p} K^{\vee} \otimes \det K$ is $(p+1)$-regular, and hence, $\gamma_k'$ is surjective for $k \geq p$.
\end{proof}

It only remains to prove that the composition (\ref{eq:claim}) is surjective. The second map in (\ref{eq:claim}) is just the globalization of the map
$$
\wedge^2 V \longrightarrow  V_h,~~x^{(i)} \wedge x^{(j)} \longmapsto a_j x^{(i)} - a_i x^{(j)},
$$
where $V_h$ is the kernel of a nonzero linear functional 
$h:=\sum_{i=0}^{p+2} a_i y^i \in V^{\vee} = S^{p+2} U^{\vee}$. We now regard $h=\sum_{i=0}^{p+2} b_i x^i \in S^{p+2}U$, where $b_i = (-1)^i a_{p+2-i}$.
The surjectivity of (\ref{eq:claim}) is equivalent to the injectivity of the composition of the dual maps
\begin{equation}\label{eq:injforduals}
V_h^{\vee} \longrightarrow \wedge^2 V^{\vee} \xrightarrow{~\Delta_{p+2}^{\vee}~} W^{\vee}.
\end{equation}
The image of $V_h^{\vee}$ in $\wedge^2 V^{\vee}$ is spanned by $v_j:=x^j \wedge h$ for $j=0,\ldots, d-1, d+1, \ldots, p+2$, where $d:=\deg h$. Recall that $\Delta_{p+2}^{\vee}$ is the restriction map $H^0(\nP^2, \sO_{\nP^2}(p+1)) \to H^0(\overline{S}, \sO_{\overline{S}}(p+1))$.

\begin{proof}[Proof of Theorem \ref{thm:AFPRW}]
We consider the case $\overline{S}=Q$. In this case, $\Delta_{p+2}^{\vee} =  \mu_{p+2}$ (see (\ref{eq:wahlmap})).
We have $\Delta_{p+2}^{\vee}(v_j)= \sum_{i=0}^d (j-i)b_i x^{i+j-1}$.
As $\operatorname{char}(\kk) = 0$ or $\operatorname{char}(\kk) \geq (g+2)/2 \geq p+3$, we get $\deg (\Delta_{p+2}^{\vee}(v_j))= d+j-1$. This implies that 
$$
\text{$\Delta_{p+2}^{\vee}(v_0), \ldots, \Delta_{p+2}^{\vee}(v_{d-1}), \Delta_{p+2}^{\vee}(v_{d+1}), \ldots, \Delta_{p+2}^{\vee}(v_{p+2})$ are linearly independent.}
$$
Thus the composition (\ref{eq:injforduals}) is injective.
 \end{proof}

\begin{proof}[Proof of Theorem \ref{thm:K3carpet}]
We consider the case $\overline{S}=2\ell$. In this case, $\Delta_{p+2}^{\vee}$ factors through $\tau_{p+2}$ (see (\ref{eq:tau})).
Write $\tau_{p+2}(v_j) = (u_j, w_j)$. By Proposition \ref{prop:ACM} $(2)$, we only need to deal with the case $p \geq 1$. As $\operatorname{char}(\kk) = 0$ or $\operatorname{char}(\kk) \geq \lfloor (g-1)/2 \rfloor \geq (p+3)/2$, we have:
$$
\begin{array}{l}
\text{$\deg u_j = d+j-1$ and $\deg w_j \leq d+j-1$ when $d \equiv j~(\operatorname{mod} 2)$;}\\
\text{$\deg u_j \leq d+j-2$ and $\deg w_j = d+j$ when $d \not\equiv j~(\operatorname{mod} 2)$.}
\end{array} 
$$
This implies that 
$$
\text{$\tau_{p+2}(v_0), \ldots, \tau_{p+2}(v_{d-1}), \tau_{p+2}(v_{d+1}), \ldots, \tau_{p+2}(v_{p+2})$ are linearly independent.}
$$
Thus the composition (\ref{eq:injforduals}) is injective.
\end{proof}

\begin{remark}\label{rem:charassump}
The characteristic assumption in Theorem \ref{thm:AFPRW} is used to prove the injectivity of (\ref{eq:injforduals}). If $k:=\operatorname{char}(\kk) \leq p+2$, then (\ref{eq:injforduals}) is not injective for $h=1$. Indeed, $x^k \wedge 1 \neq 0$ in $\wedge^2 V^{\vee}$ is sent to $0$ in $W^{\vee}$. The characteristic assumption in Theorem \ref{thm:K3carpet} is used in Lemma \ref{lem:vankosmod}.
\end{remark}

\begin{remark}\label{rem:K_{p,1}}
By the duality theorem, $K_{p,q}(S, \sO_S(1)) = K_{g-p-2, 3-q}(S, \sO_S(1))^{\vee}$. 
In the situation of Theorem \ref{thm:AFPRW} or Theorem \ref{thm:K3carpet}, we have
$$
K_{p,1}(S, \sO_S(1))=0~~\text{ for $p \geq \lfloor g/2 \rfloor$}.
$$
As $K_{p+1, 1}(S, \sO_S(1)) =\ker(\rho_{p+2})= \ker(\gamma_{p+2})$ and $\gamma_{p+2}$ is surjective for $0 \leq p \leq \lfloor (g-3)/2 \rfloor$, we can compute $K_{p, 1}(S, \sO_S(1))$ for $0 \leq p \leq \lfloor (g-1)/2 \rfloor$. Note that $K_{p,0}(S, \sO_S(1)) \neq 0$ if and only if $p=0$. In this case, $\kappa_{0,0}(S, \sO_S(1))=1$. Thus we can completely determine all graded Betti numbers $\kappa_{p,q}(S, \sO_S(1))$.\end{remark}

%%%%%%%%%%%%%%%%%%%%%%%%%%%%%%%%%%%%%%%%%%%%%%%%%%
\section{Proof of Theorem \ref{thm:arithnorm}}\label{sec:proof2}
%%%%%%%%%%%%%%%%%%%%%%%%%%%%%%%%%%%%%%%%%%%%%%%%%%

Assume that $\operatorname{char}(\kk) = 0$. Let $C \subseteq \nP^r$ be a smooth projective curve of genus $g \geq 1$ embedded by $|L|$, where $L$ is a line bundle on $C$ with $\deg L \geq 2g+3$. We use the notations in Section \ref{sec:prelim}. Recall from Theorem \ref{thm:secant} that $\Sigma \subseteq \nP^r$ is arithmetically Cohen--Macaulay and $H^3(\Sigma, \sO_{\Sigma}(m)) = 0$ for $m > 0$. By considering (\ref{eq:sesOSgeneral1}) and Lemma \ref{lem:R1O_B(-S-Z)}, we see that
\begin{equation}\label{eq:H^i(O_T(m))+norm}
\begin{gathered}
\begin{array}{l}
H^{i}(T, \sO_T(m))=H^{i+1}(B, \sO_B(m H-\widetilde{T}-Z)) ~~\text{ for $i \geq 1, m \geq 1$};\\
\text{$T \subseteq \nP^r$ is $m$-normal} ~~\Longleftrightarrow~~H^1(B, \sO_B(m H-\widetilde{T}-Z))=0.
\end{array}
\end{gathered}
\end{equation}
Note that $m H - \widetilde{T} - Z = (m-2)H + \pi^* T_L(-4\delta)$. Thus $R^1 \pi_*\sO_B(m H - \widetilde{T} - Z) = 0$ for $m \geq 1$. When $m=1$, we have  $H^j(B, \sO_B(H-\widetilde{T}-Z))=0$ for $j > 0$ since $\pi_*\sO_B(H-\widetilde{T}-Z) = 0$.
When $m=2$, we have $H^j(B, \sO_B(2H-\widetilde{T}-Z))=H^j(C_2, T_L(-4\delta))$ for $j> 0$. Recall that $\sigma_*\sO_Z=\sO_{C_2} \oplus \sO_{C_2}(-\delta)$. Then we have $\sigma_* (L \boxtimes L)(-kD) = T_L(-(k+1)\delta) \oplus T_L(-k\delta)$. 

\begin{lemma}\label{lem:regO_T}
$H^i(T, \sO_T(m))=0$ for $i>0, m>0$, and $h^1(T, \sO_T) = 2g, ~h^2(T, \sO_T)=1$. In particular, $\reg \sO_T = 3$.
\end{lemma}

\begin{proof}
By (\ref{eq:sesOSgeneral2}) and Proposition \ref{prop:omegaS}, $h^1(T, \sO_T) = 2g$ and $h^2(T, \sO_T)=1$. It suffices to prove that $H^i(T, \sO_T(m)) = 0$ for $i=1,m=1,2$ and $i=2,m=1$. Indeed, this implies $\reg \sO_T = 3$. By (\ref{eq:H^i(O_T(m))+norm}), we need to check that $H^{j}(B, \sO_B(m H-\widetilde{T}-Z))=0$ for $j=2,m=1,2$ and $j=3,m=1$. As the required vanishing is trivial when $m=1$, it is enough to show that $H^2(C^2, (L \boxtimes L)(-3D))=0$. Observe that $R^1 p_* (L \boxtimes L)(-3D) = 0$, where $p \colon C \times C \to C$ is a projection map. Then $H^2(C^2, (L \boxtimes L)(-3D)) = H^2(C, p_* (L \boxtimes L)(-3D))=0$.
\end{proof}

As $T \subseteq \nP^r$ is linearly normal and $h^1(T, \sO_T)=2g$, a general hyperplane section of the tangent developable surface $T \subseteq \nP^r$ is obtained from an isomorphic projection of an $(r+g)$-cuspical curve of geometric genus $g$ canonically embedded in $\nP^{r+2g}$.

\begin{lemma}\label{lem:2-normality}
Suppose that $\deg L \geq 3g+2$. Then the $2$-normality of $T \subseteq \nP^r$ is equivalent to that $H^1(C \times C, (L \boxtimes L)(-3D))=0$.
\end{lemma}

\begin{proof}
By (\ref{eq:H^i(O_T(m))+norm}), the $2$-normality of $T \subseteq \nP^r$ is equivalent to that $H^1(C_2, T_L(-4\delta))=0$. By \cite[Theorem 1 (i)]{BEL}, $H^1(C^2, (L \boxtimes L)(-2D))=0$. It follows that $H^1(C_2, T_L(-3\delta))=0$. Then  $H^1(C_2, T_L(-4\delta))=0$ if and only if $H^1(C^2, (L \boxtimes L)(-3D))=0$. 
\end{proof}

\begin{proposition}\label{prop:K_{p,3}}
$K_{p,3}(T, \sO_T(1))=0$ for $0 \leq p \leq r-3$, and $\kappa_{r-2,3}(T, \sO_T(1))=1$.
\end{proposition}

\begin{proof}
By Proposition \ref{prop:koszulcoh} and Lemma \ref{lem:regO_T}, $K_{p,3}(T, \sO_T(1))=H^2\big(T, \wedge^{p+2}M_{\sO_T(1)} \otimes \sO_T(1)\big)$. Note that $\omega_T = \sO_T$ (Proposition \ref{prop:omegaS}) and $\wedge^{p+2} M_{\sO_T(1)}^{\vee} = \wedge^{r-p-2} M_{\sO_T(1)} \otimes \sO_T(1)$. By Serre duality, we have
$$
H^2\big(T, \wedge^{p+2}M_{\sO_T(1)} \otimes \sO_T(1)\big) = H^0\big(T, \wedge^{p+2} M_{\sO_T(1)}^{\vee} \otimes \sO_T(-1)\big)^{\vee} = H^0\big(T, \wedge^{r-p-2} M_{\sO_T(1)}\big)^{\vee}.
$$
Then the proposition immediately follows.
\end{proof}

\begin{proof}[Proof of Theorem \ref{thm:arithnorm}]
By Lemma \ref{lem:regO_T}, it suffices to show that $T \subseteq \nP^r$ is $m$-normal for $m=1,2,3$. Indeed, this together with $\reg \sO_T = 3$ implies $\reg \sI_{T|\nP^r} = 4$, and thus, $T \subseteq \nP^r$ is $m$-normal for $m \geq 1$. In view of (\ref{eq:H^i(O_T(m))+norm}), $T \subseteq \nP^r$ is trivially $1$-normal. For the $2$-normality, we assume that $\deg L \geq 4g+3$. Then \cite[Theorem 1.7 (i)]{BEL} says that $H^1(C \times C, (L \boxtimes L)(-3D))=0$, so the $2$-normality of $T \subseteq \nP^r$ follows from Lemma \ref{lem:2-normality}. Now, the $3$-normality of $T \subseteq \nP^r$ is equivalent to that  $K_{0,3}(T, \sO_T(1))=0$, but this is a special case of Proposition \ref{prop:K_{p,3}}.
\end{proof}

\begin{remark}\label{rem:degL=4g+2}
In Theorem \ref{thm:arithnorm}, the degree condition $\deg L \geq 4g+3$ is only used to show that $T \subseteq \nP^{r}$ is $2$-normal. If $\deg L =4g+2$, then by Lemma \ref{lem:2-normality} and \cite[Theorem 1.7 (ii), (iii)]{BEL}, $T \subseteq \nP^r$ is $2$-normal (and arithmetically normal) if and only if $g \geq 3$ and $C$ is nonhyperelliptic. More generally, using the results in \cite[Section 1]{BEL} and \cite[Theorem 3.8 (b)]{Pareschi}, one can prove that if $\deg L \geq 4g+3-\Cliff(C)$, then $T \subseteq \nP^r$ is arithmetically normal.
\end{remark}

\begin{corollary}
Suppose that $\deg L \geq 4g+3$. Then the Hilbert function of $T \subseteq \nP^r$ is given by $H_T(t) = (\deg L + g-1)t^2 +2-2g$ for $t \geq 1$. In particular, $\kappa_{1,1}(T, \sO_T(1))=(r-2)(r-3)/2-6g$.
\end{corollary}

\begin{proof}
By Theorem \ref{thm:arithnorm}, $H_T(t)$ coincides with the Hilbert polynomial $P_T(t)$ for $t \geq 1$. 
We know that $\deg P_T(t)=2$ and the leading coefficient of $P_T(t)$ is $(\deg T)/2 = \deg L + g-1$.
As $P_T(0)=\chi(\sO_T) = 2-2g$ and $P_T(1) = h^0(C, L)= \deg L - g+1$, we easily get the assertion. The last assertion follows from that $\kappa_{1,1}(T, \sO_T(1))=h^0(\nP^r, \sI_{T|\nP^r}(2)) = {r+2 \choose 2} - H_T(2)$.
\end{proof}

Observe that the tangent developable surface of an elliptic normal curve of degree $d$ has the same Hilbert polynomial with an abelian surface of degree $2d$ in $\nP^{d-1}$.

\begin{remark}\label{rem:syztansurf}
Assume that $\deg L \geq 4g+3$. Notice that (cf. Proposition \ref{prop:K_{p,3}})
$$
\begin{array}{rcl}
K_{p,0}(T, \sO_T(1)) \neq 0 & \Longleftrightarrow & p=0~\text{ }~(\kappa_{0,0}(T, \sO_T(1))=1);\\
K_{p,3}(T, \sO_T(1)) \neq 0 & \Longleftrightarrow & p=r-2~\text{ }~(\kappa_{r-2,3}(T, \sO_T(1))=1).
\end{array}
$$
It would be exceedingly interesting to study vanishing/nonvanishing behaviors of $K_{p,1}(T, \sO_T(1))$ and $K_{p,2}(T, \sO_T(1))$ as the positivity of $L$ grows. For a possible generalization of Theorem \ref{thm:AFPRW} (see also Remark \ref{rem:K_{p,1}}), one may expect that
\begin{equation}\label{eq:expectK_{p,1}}
K_{p,1}(T, \sO_T(1)) = 0 ~~\text{ for $p \geq \lfloor (\deg L-3g)/2 \rfloor$}.
\end{equation}
We have $\kappa_{p,1}(T, \sO_T(1)) \leq \kappa_{p,1}(C, L)$. Thus $K_{p,1}(T, \sO_T(1)) = 0$ at least for $p \geq r-1$ (or for $p \geq r-\gon(C)+1$ by \cite[Theorem 1.1]{Rathmann}). Then we find
$$
\kappa_{r-2,2}(T, \sO_T(1))=2g(r+1),~\kappa_{r-1,2}(T, \sO_T(1))= 2g.
$$
If $(\ref{eq:expectK_{p,1}})$ holds, then we can compute $\kappa_{p,2}(T, \sO_T(1))$ for $\lfloor (\deg L - 3g -2)/2 \rfloor \leq p \leq r-1$. Concerning vanishing of $K_{p,2}(T, \sO_T(1))$, one may ask whether $K_{p,2}(T, \sO_T(1))=0$ as soon as $\deg L \geq 4g+3+2p$. Example \ref{ex:Bettitable} shows that this is not true for $g=1$ and $p=1$ but true for $g=2$ and $p=1$. When $g=1$, it seems likely that
\begin{equation}\label{eq:expectK_{p,2}}
K_{p,2}(T, \sO_T(1)) = 0 ~~\text{ for $p \leq \lfloor (\deg L - 7)/3 \rfloor$}.
\end{equation}
In general, it is tempting to guess that $K_{p,2}(T, \sO_T(1)) = 0$ for $p \leq \lfloor (\deg L - 4g-3)/(g+1) \rfloor$, but it is hard to make a precise prediction at this moment.
\end{remark}

\begin{example}\label{ex:Bettitable}
Let $q_1, \ldots, q_m$ be minimal generators of the ideal $I_{C|\nP^r}$. Then $\widetilde{T}$ is defined by $q_i$ and $\sum_{j=0}^r (\partial q_i/ \partial x_j) y_j$ for $1 \leq i \leq m$ in $\nP^r \times \nP^r$, and $T=q(\widetilde{T}) \subseteq \nP^r$, where $q \colon \nP^r \times \nP^r \to \nP^r$ is the second projection. Using \texttt{Macaulay2} \cite{GS}, we can compute the Betti table of $T$ in $\nP^r$ whenever $r$ is reasonably not so big.

\smallskip

\noindent $(1)$ Let $C := Z(z_0^3 - z_0 z_2^2 - z_1^2 z_2) \subseteq \nP^2$ be an elliptic curve. There is a point $x \in C$ such that $\sO_C(1)=\sO_C(3x)$. Consider the embedding $C \subseteq \nP^{d-1}$ given by $|\sO_C(dx)|$ for $d \geq 3$. The computation of the Betti table of the tangent developable surface $T$ of $C$ in $\nP^{d-1}$ confirms (\ref{eq:expectK_{p,1}}) and (\ref{eq:expectK_{p,2}}) when $7 \leq d \leq 16$. For these cases, the predictions are sharp. The Betti tables of $T$ in $\nP^{d-1}$ for $d=9,10$ are as follows.
\begin{figure}[!htb]
\begin{subfigure}{.45\textwidth}
\centering
\texttt{\begin{tabular}{l|cccccccc}
          & 0 & 1 &  2  &3   & 4 & 5 & 6 & 7 \\ \hline
    0    & 1 & -  & -    &-   & -  & -  & -  & - \\
    1   & -  & 9 & 3    & -  & -  & -  & -  & - \\
    2   & -  & 6 & 81  & 171 & 165 & 81 & 18 & 2\\
    3    & -  & - & - &   -  & -  & -  & 1 & - 
    \end{tabular}}
\end{subfigure}
\begin{subfigure}{.498\textwidth}
\centering
\texttt{\begin{tabular}{l|ccccccccc}
          & 0 & 1 &  2  &3   & 4 & 5 & 6 & 7 & 8 \\ \hline
    0    & 1 & -  & -    &-   & -  & -  & -  & - & - \\
    1   & -  & 15 & 20    & -  & -  & -  & -  & -  & -\\
    2   & -  & - & 70  &252 & 350 & 260 & 105 & 20 & 2\\
    3    & -  & - & - &   -  & -  & -  & - & 1 & -
    \end{tabular}}
    \end{subfigure}
    \label{Bettitable}
\end{figure}

\noindent $(2)$ Let $C \subseteq \nP^1 \times \nP^1$ be a smooth projective curve of genus $2$ defined by 
$$
x^2 \otimes (x^3+1) + 1 \otimes (x^2-x) \in S^2U \otimes S^3 U = H^0(\nP^1 \times \nP^1, \sO_{\nP^1}(2) \boxtimes \sO_{\nP^1}(3)).
$$
Consider the embedding $C \subseteq \nP^{11}$ given by $|(\sO_{\nP^1}(3) \boxtimes \sO_{\nP^1}(2))|_C|$.  Note that $\deg C = 13$. The Betti table of the tangent developable surface $T$ of $C$ in $\nP^{11}$ is the following.

\bigskip

\texttt{ \begin{tabular}{l|ccccccccccc}
         & 0 & 1 &  2  &3   & 4 & 5 & 6 & 7 & 8 & 9 & 10\\ \hline
    0    & 1 & -  & -    &-   & -  & -  & -  & - & - & -  & -\\
   1   & -  & 24 & 48    & -  & -  & -  & -  & -  & - & - & -\\
    2   & -  &  - & 153  & 864 & 1848 & 2304 & 1827 & 928 & 288 & 48 & 4 \\
    3    & -  & - & - &   -  & -  & -  & - & -  & - & 1 & -
\end{tabular}}
\end{example}

%%%%%%%%%%%%%%%%%%%%%%%%%%%%%%%%%%%%%%%%%%%%%%%%%%
\bibliographystyle{ams}
%%%%%%%%%%%%%%%%%%%%%%%%%%%%%%%%%%%%%%%%%%%%%%%%%%

\end{document}